\documentclass[12pt,leqno,a4paper]{amsart}
\usepackage{amssymb}

\newcommand{\FF}{{\mathbb{F}}}

\newcommand{\bG}{{\mathbf{G}}}
\newcommand{\bL}{{\mathbf{L}}}

\newcommand{\fA}{{\mathfrak{A}}}
\newcommand{\fS}{{\mathfrak{S}}}

\newcommand{\cE}{{\mathcal{E}}}
\newcommand{\cS}{{\mathcal{S}}}

\newcommand{\Aut}{{\operatorname{Aut}}}
\newcommand{\Irr}{{\operatorname{Irr}}}
\newcommand{\Syl}{{\operatorname{Syl}}}
\newcommand{\PSL}{{\operatorname{PSL}}}
\newcommand{\GL}{{\operatorname{GL}}}
\newcommand{\SL}{{\operatorname{SL}}}
\newcommand{\GU}{{\operatorname{GU}}}
\newcommand{\SU}{{\operatorname{SU}}}
\newcommand{\Sp}{{\operatorname{Sp}}}
\newcommand{\SO}{{\operatorname{SO}}}
\newcommand{\OO}{{\operatorname{O}}}

\newcommand{\tw}[1]{{}^#1\!}
\newcommand\flsq[1]{{\lfloor\sqrt{#1}\rfloor}}
\newcommand\hlf{{\frac{1}{2}}}

\let\nl=\ell
\let\al=\alpha
\let\la=\lambda
\let\si=\sigma

\newtheorem{thm}{Theorem}[section]
\newtheorem{lem}[thm]{Lemma}
\newtheorem{prop}[thm]{Proposition}
\newtheorem{cor}[thm]{Corollary}

\newtheorem{thmB}{Theorem}
\newtheorem{conjB}[thmB]{Conjecture}
\newtheorem{corB}[thmB]{Corollary}

\theoremstyle{remark}
\newtheorem{rem}[thm]{Remark}

\begin{document}

\title[Simple modules in a block]{On the number of simple modules\\ in a block of a finite group}

\date{\today}

\author{Gunter Malle and Geoffrey R. Robinson}
\address{FB Mathematik, TU Kaiserslautern, Postfach 3049,
         67653 Kaisers\-lautern, Germany.}
\email{malle@mathematik.uni-kl.de}
\address{Institute of Mathematics, Fraser Noble Building,
         University of Aberdeen, Aberdeen AB24 3UE, United Kingdom.}
\email{g.r.robinson@abdn.ac.uk}

\thanks{The first author gratefully acknowledges financial support by ERC
  Advanced Grant 291512.}



\dedicatory{To the memory of Sandy Green}
\begin{abstract}
We prove that if $B$ is a $p$-block with non-trivial defect group $D$ of a
finite $p$-solvable group $G$, then $\nl(B) < p^r$, where $r$ is the
sectional rank of $D$. We remark that there are infinitely many $p$-blocks
$B$ with non-Abelian defect groups and $\nl(B) = p^r - 1$.

\medskip
We conjecture that the inequality $\nl(B) \leq p^r$ holds for an arbitrary
$p$-block with defect group of sectional rank $r$. We show this to hold for
a large class of $p$-blocks of various families of quasi-simple and nearly
simple groups.
\end{abstract}

\maketitle

\section{Introduction} \label{sec:intro}

Almost 60 years ago, Richard Brauer posed his fundamental problem on $k(B)$,
the number of ordinary irreducible characters in a $p$-block $B$ of defect $d$
of a finite group: Is it always the case that $k(B) \leq p^d$? Despite partial
progress, this question remains open to the present day. Here, we propose a
conjecture of a similar flavour for the number of irreducible Brauer characters
in a $p$-block. In order to formulate it, we need the following notion:

For a $p$-block $B$ let $s(B)$ denote the sectional $p$-rank of a defect group
$D$ of $B$, that is, the largest rank of an elementary Abelian section of $D$.
Further, let $\nl(B)$ denote the number of isomorphism classes of simple
modules in $B$.

\begin{conjB}   \label{conj:main}
 Let $B$ be a $p$-block of a finite group. Then $\nl(B)\le p^{s(B)}$.
\end{conjB}

Let us mention that in all of our results, the inequality in
Conjecture~\ref{conj:main} turns out to be strict for blocks of positive
defect, unless possibly when $p=3$ and $B$ has extra-special defect group
of order~27, see Remark~\ref{rem:27}.

The first main result of this paper is the proof of our conjecture (with strict
inequality) in the case of $p$-solvable groups, which will be given in
Section~\ref{sec:psolv} using the positive solution of the $k(GV)$-problem:

\begin{thmB}   \label{thm:main}
 Let $B$ be a $p$-block of a finite $p$-solvable group with non-zero
 defect. Then $\nl(B)< p^{s(B)}$.
\end{thmB}

This has the following consequence:

\begin{corB}   \label{cor:main}
 Let $B$ be a block with an abelian defect group, and assume that $B$ either
 satisfies Alperin's weight conjecture, or Brou\'e's perfect isometry
 conjecture. Then Conjecture~\ref{conj:main} holds for $B$, with strict
 inequality if $B$ has non-zero defect.
\end{corB}

Our second main result concerns certain blocks of nearly simple groups.
More precisely we show the following:

\begin{thmB}   \label{thm:main2}
 Conjecture~\ref{conj:main} holds for the following $p$-blocks of nearly simple
 groups:
 \begin{enumerate}
  \item[\rm(a)] all $p$-blocks of all covering groups of symmetric,
   alternating, sporadic, special linear and special unitary groups;
  \item[\rm(b)] all $p$-blocks of all quasi-simple groups of Lie type in
   characteristic~$p$;
  \item[\rm(c)] all unipotent $p$-blocks of quasi-simple groups $G$ of
   Lie type in characteristic different from $p$ when $p$ is good for $G$; and
  \item[\rm(d)] the principal $p$-blocks of all quasi-simple groups of Lie type.
 \end{enumerate}
\end{thmB}

This is shown in Propositions~\ref{prop:alt}, \ref{prop:spor},
\ref{prop:Liedef}, Theorems~\ref{thm:SLSU}, \ref{thm:Lieunip} and
Proposition~\ref{prop:unipexc}
respectively. The proofs in those cases involve interesting combinatorial
questions on numbers of unipotent characters and of unipotent classes in
groups of Lie type.

In view of Conjecture~\ref{conj:main} it would be interesting to know whether
or under what conditions the sectional rank of a defect group is an invariant
under Morita equivalences of blocks.

We also have some partial results on minimal counter examples to
Conjecture~\ref{conj:main} with a non-trivial normal $p$-subgroup, which we
hope to pursue in future work.

\section{Proof of the $p$-solvable case}   \label{sec:psolv}
We recall for the convenience of the reader that the \emph{sectional rank} of
a finite group $G$ is the maximum over all sections of $G$ of the minimum
number of generators of the section.
The trivial group is considered here to have sectional rank $0$. Since every
section of $G$ is a homomorphic image of a subgroup of $G$, the sectional rank
of $G$ is also the maximum over all subgroups of $G$ of the minimal number of 
generators of that subgroup. When $G$ is a finite $p$-group for some prime $p$,
the sectional rank of $G$ is $r$,
where $p^r$ is the order of the largest elementary Abelian section of $G$.
For a $p$-block $B$ we let $s(B)$ denote the sectional $p$-rank of any defect
group of $B$.

\begin{proof}[Proof of Theorem~\ref{thm:main}]
Let $B$ be a $p$-block of a finite $p$-solvable group $G$ such that
$s(B)=r\ge1$. We wish to prove that $\nl(B) < p^{r}$, where $\nl(B)$ is the
number of simple modules in $B$. By Fong's first reduction, we may suppose
that $B$ is a primitive block. Hence, whenever $N \lhd G$ has order prime
to $p$, $\Irr(B)$ covers a single irreducible character of $N$. By
Fong's second reduction, we may then suppose that $O_{p^\prime}(G)\leq Z(G)$.
Let $Z = Z(G)$ and let $\lambda \in \Irr(Z)$ lie under $B.$ Note that
we now have $D \in \Syl_p(G).$

\medskip
Let $U = O_p(G)$ and let $H$ be a Hall $p^\prime$-subgroup of $G$. Then
there are $\nl(B)$ projective indecomposable characters of $G$ lying over
$\lambda$, each of which is induced from an irreducible character of $H$
lying over $\lambda.$ Hence $\nl(B) \leq k(H|\lambda) \leq k(H/Z).$

The following lemma (stated in somewhat greater generality than required here)
now completes the proof.
\end{proof}

\begin{lem}   \label{lem:k(GV)}
 Let $X$ be a non-trivial finite $p$-solvable group with $O_{p^\prime}(X)=1$
 such that $U = O_p(X)$ has sectional rank $r.$ Then a Hall
 $p^\prime$-subgroup of $X$ has fewer than $p^{r}$ conjugacy classes. In
 particular, the number of $p$-regular conjugacy classes of $X$ is less
 than~$p^r$.
\end{lem}

\begin{proof}
Let $V = U/\Phi(U)$. Then $Y = X/U$ acts faithfully as a group of linear
transformations of $V$ and $|V| \leq p^{r}.$ Let $T$ be a Hall
$p^\prime$-subgroup of $Y.$ Then the number of $p$-regular conjugacy classes
of $X$ is the number of $p$-regular conjugacy classes of $Y$, which is at most
the number of conjugacy classes of $T$. Now $k(T) < k(VT) \leq |V| \leq p^{r}$,
using the solution of the $k(GV)$-problem \cite{kGV}.
\end{proof}

\begin{rem}
Note that the inequality $k(B) \leq p^{s(B)}$ need not hold (consider $\fS_4$
for $p=2$, $B$ its unique $2$-block). Also, there are $2$-blocks $B$ of
solvable groups with dihedral defect group of order $8$ with
$\nl(B) = 3 = 2^{s(B)}-1$, so the given upper bound in Theorem~\ref{thm:main}
for the strict inequality for $\nl(B)$ can be attained.
\end{rem}

\section{Remarks on Conjecture~\ref{conj:main}}
We remark that it is not possible to replace ``sectional rank" by ``the
maximal rank of an Abelian subgroup" in Conjecture~\ref{conj:main} (consider
the principal $2$-block of $\SL_2(3)$, with quaternion defect group, for
example).

\medskip
We note that Conjecture~\ref{conj:main} gives a somewhat less precise statement
than that predicted by Alperin's Weight Conjecture. It is not clear to us
whether an affirmative answer to Conjecture~\ref{conj:main} would follow from
an  affirmative answer to Alperin's Weight Conjecture.

\medskip
In what follows, we will say that Conjecture~\ref{conj:main} holds \emph{in
strong form} for $B$ if we have $\nl(B) < p^{s(B)}$. We have seen that in the
case that $B$ is a $p$-block of positive defect of a $p$-solvable group then
Conjecture~\ref{conj:main} holds in strong form for $B$.

\begin{prop}
 Conjecture~\ref{conj:main} holds in strong form for blocks with non-trivial
 cyclic defect, as well as for 2-blocks with dihedral, (generalised) quaternion
 or semi-dihedral defect groups.
\end{prop}

\begin{proof}
This is clear for $p$-blocks with cyclic defect groups, since any such block
has at most $p-1$ simple modules.
If $B$ is a 2-block with dihedral defect group $D$ then
$\nl(B) \leq 3<4=2^{s(B)}$.  For 2-blocks $B$ with (generalised) quaternion
defect group $D$, we have $\nl(B) = \nl(b)$, with  $b$ the unique Brauer
correspondent of $B$ for $C_G(t)$, where $t$ is the central involution of
$D$. Furthermore, we have $\nl(b) = \nl({\tilde b})$, where ${\tilde b}$ is
the unique block of $C_G(t)/\langle t \rangle$ dominated by $b$, and
${\tilde b}$ has a dihedral (possibly including Klein 4) defect group.
Finally, for any $2$-block $B$  with semi-dihedral defect groups we
again always have $\nl(B) \leq 3$ (see e.g.~\cite[Thm.~8.1]{S14}).
\end{proof}

As stated in Corollary~\ref{cor:main} there is a strong relation to other
standard conjectures in block theory.

\begin{proof}[Proof of Corollary~\ref{cor:main}]
If either of Brou\'e's perfect isometry conjecture or Alperin's Weight
Conjecture holds for a block $B$ with abelian defect groups, then we have
$\nl(B) = \nl(b)$ where $b$ is the unique Brauer correspondent of $B$ for
the normaliser $N_G(D)$ of a defect group $D$ of $B$. By theorems of
K\"ulshammer and Reynolds (see \cite[Thms.~1.18, 1.19]{S14}, for example),
$b$ is Morita equivalent to a block ${\tilde b}$ of a finite group $H$ with
normal Sylow $p$-subgroup $D$, so it satisfies Conjecture~\ref{conj:main}
in strong form.
\end{proof}


As with Brauer's $k(B) \leq |D|$ question, Clifford theoretic reductions for
Conjecture~\ref{conj:main} are not entirely straightforward (in fact, they are
opaque at present). We concentrate on the minimal case.

\begin{lem}   \label{lem:min}
 If $B$ is a $p$-block with $\nl(B) > p^{s(B)}$ of a finite group $G$ of
 minimal order subject to this occurring, then we have
 $O_{p^\prime}(G) = Z(G)$, as well as $G = O^p(G)$ and $O_p(G)$
 elementary Abelian.
\end{lem}

\begin{proof}
The fact that $G = O^p(G)$ follows since $B$ is necessarily primitive and
$B$ covers a unique block $b$ of $O^p(G)$. For every projective
indecomposable of $B$ is induced from a projective indecomposable of $b$.
Hence $\nl(B) \leq \nl(b)$, while the defect group of $B$ contains the defect
group of $b$. The fact that $U = O_p(G)$ is elementary Abelian follows,
since by \cite[Cor.~7]{Ro83}, $B$ dominates a unique block $\tilde B$
of $G/\Phi(U)$ with defect group $D/\Phi(U)$, and $\nl(B) = \nl({\tilde B})$.
The fact that $O_p(Z(G)) = 1$ follows for the same reason. By Fong's second
reduction we have that $O_{p'}(G)\le Z(G)$.
\end{proof}

\begin{prop}
 Let $b$ be a $p$-block of $G$ of positive defect satisfying
 Conjecture~\ref{conj:main} and $C$ a cyclic group of prime order $q\neq p$.
 Then any $p$-block $B$ of $G \wr C$ which covers $b \otimes \ldots \otimes b$
 ($q$ factors) satisfies the strong form of Conjecture~\ref{conj:main}.
\end{prop}

\begin{proof}
As $b$ has positive defect, $B$ has at most
$\frac{\nl^q - \nl}q + q \nl$ simple modules when $b$ has $\nl$ simple
modules. This is less than $\nl^q$ when $\nl >2,$ so we may suppose that
$\nl \leq 2$. For note that the sectional rank of a defect group of $B$ is
$qr$, where $r$ is the sectional rank of a defect group of $b$.
If we have $\frac{2^q-2}q+2q \geq p^{rq}$ then $2^q + 2q^2-2 \geq qp^{rq}$.
In particular, $2^q \leq  2(q+1)$ which forces $q \leq 3$. If $q =2$ then
$p^{2r} \leq 5$ so $ p = 2,$ contrary to the fact that $p \neq q$. If $q = 3$,
we have $p^{3r} \leq 8$, so $p =2, r =1$. But then $b$ has cyclic defect
($2$-)group, so we have $\nl = 1$, and there are just $3$ simple modules in
$B$, while $2^{3r} = 8$.
\end{proof}

Similarly, our conjecture is compatible with direct products. For this note
that the sectional $p$-rank of a direct product of finite $p$-groups is
the sum of the sectional $p$-ranks of the factors.

\section{Implications of the conjecture for Brauer's question}
Recall Brauer's question whether it is always true that $k(B) \leq p^{d}$
for a $p$-block of defect $d$. The best known general bounds for $k(B)$ are
given by Brauer and Feit in \cite{BF59}. They show that $k(B) \leq p^{2d-2}$.

\medskip
Here we note that a proof of our conjecture on $\nl(B)$ would imply bounds
of a different nature for a block $B$ with defect group $D$ of order $p^{d}$
and sectional rank $r$. For it is a consequence of Brauer's Second Main
Theorem that $k(B) = \sum_{x} \sum_{b} \nl(b)$, where $x$ runs through
a set of representatives for the $G$-conjugacy classes of elements of $D$
and $b$ runs through blocks of $C_G(x)$ with $b^G = B$.

Using the Alperin--Brou\'e theory of $B$-subpairs, we may
write this expression as $k(B) = \sum_{(x,b)} \nl(b)$, where $(x,b)$ is a
Brauer element contained in the maximal $B$-subpair $(D,b)$, and these
are taken up to $G$-conjugacy.

If Conjecture~\ref{conj:main} is correct, then we would certainly have
$k(B) \leq \sum_x\sum_b p^r$ where $x$ and $b$ are as previously.
So $k(B) \leq p^{r} k_G(D,b)$, where $k_G(D,b)$ is the number of
$G$-conjugacy classes of Brauer elements contained in $(D,b)$. In any case,
there will be many blocks for which this can be sharpened, for it is quite
likely that $s(b) < s(B)$ for many of the blocks~$b$.

\begin{prop}
 Let $G$ be a finite group with $p$ dividing $|G|$. Assume that
 Conjecture~\ref{conj:main} holds for the principal $p$-block of the
 centraliser $C_G(x)$ for each $p$-element $x\in G$. Then for the principal
 block $B$ of $G$ we have $k(B) \leq p^{r}(k(S)-1)$, where $S$ is a
 Sylow $p$-subgroup of $G$.
\end{prop}

\begin{proof}
As $B$ is the principal block, using Brauer's Third Main Theorem the above
considerations show that our conjecture implies $k(B) \leq p^{r} k_G(S)$,
where $k_G(S)$ is the number of
$G$-conjugacy classes of $S$, and $r$ is the sectional rank of $S$. Clearly
this yields $k(B)\leq p^rk(S)$, with $k(S)$ the number of conjugacy
classes of $S$. In fact if $k_G(S) =k(S)$, then we have
$S \cap G^\prime = S^\prime$, and by a theorem of Tate, $G$ has a normal
$p$-complement (that is, the principal $p$-block of $G$ is nilpotent). Then we
have $\nl(B) = 1$ and $k(B) = k(S)$. Hence we always have
$k(B) \leq p^{r}(k(S)-1)$ when $p$ divides $|G|$.
\end{proof}

If any two Brauer elements contained in $(D,b)$ are already conjugate
via an element of $D$, then we may use the Brou\'e--Puig star construction
and follow the procedure used in \cite{BP80}: we start with an irreducible
character $\chi$ of height zero in $B$, and for each irreducible character
$\mu$ of $D$ we form the irreducible character $\chi_\ast \mu$ in $B$. Then
it follows that $k(B)=k(D)$ and $\nl(B)=1$. It follows that in all cases we should have $k(B)\leq p^{r}(k(D)-1)$ when $B$ has positive defect.

\begin{rem}
We note that $p^rk_G(D,b)$ can sometimes be less than $|D|$. For example,
if $B$ has extra-special defect groups of order $27$ and we have $k_G(D,b)=2$
(which does happen in genuine examples), then $p^r k_G(D,b) = 18$. In any
event, the weaker inequality $k(B) \leq p^r(k(D)-1)$ is likely to improve
on $k(B) \leq p^{2d-2}$ for most defect groups (though not when $D$ is
extra-special or elementary Abelian).
\end{rem}


\section{Alternating and sporadic groups} \label{sec:alt}

We now turn to verifying Conjecture~\ref{conj:main} for certain blocks of
nearly simple groups. Since the conjecture holds for blocks with cyclic
defect, we may and will restrict ourselves to considering blocks whose defect
groups are not cyclic. Moreover, by Brauer--Feit, we have $k(B)\le p^2$ for
blocks $B$ with defect group of order~$p^2$, so certainly $\nl(B)<p^2$ and
Conjecture~\ref{conj:main} holds in strong form for $B$.
Hence we may assume that $B$ has defect at least~$3$.

Secondly, assume $Z\lhd G$ is a normal subgroup of order prime to $p$. Then
any $p$-block of $G/Z$ can be considered as a $p$-block of $G$ with isomorphic
defect groups and the same number of irreducible Brauer characters. Thus,
when investigating the validity of Conjecture~\ref{conj:main} for blocks of
a quasi-simple group $G$ we may pass to the universal $p'$-covering group
of $G$ (as by Lemma~\ref{lem:min} we need not consider $p$-covers).

For integers $s,t\ge1$ let us denote by $k(s,t)$ the number of $s$-tuples of
partitions of $t$. In particular, $k(1,t)$ is the number of partitions of $t$.
We will use the following estimates:

\begin{lem}[Olsson]   \label{lem:Olsson}
 Let $s,t\ge1$. Then $k(s,t)<(s+1)^t$. If moreover $s\ge2$ then $k(s,t)\le s^t$
 unless $s=2$ and $t\le 6$.
\end{lem}

\begin{proof}
By \cite[Prop.~5]{Ol84} for all $s\ge2$ we have $k(s,t)\le s^t$ unless
$s=2$ and $t\le 6$. As $k(s,t)<k(s+1,t)$ and $k(1,t)<2^t$ for $t\le6$, all
claims follow.
\end{proof}

\begin{prop}   \label{prop:alt}
 Conjecture~\ref{conj:main} in strong form holds for the blocks of alternating
 and symmetric groups and their covering groups for all primes.
\end{prop}

\begin{proof}
The 2-blocks of the 3-fold covering groups of $\fA_6$ and $\fA_7$ can be
dealt with directly. There is nothing to check for the 6-fold covers since
Sylow $p$-subgroups for $p>3$ are cyclic. In all other cases the numbers
$\nl(B)$ have been computed by G. de B. Robinson (see \cite{Ol93}). First let
$B$ be a $p$-block of $\fS_n$ of weight $w$. Then any defect group of $B$ has
an elementary Abelian subgroup of rank~$w$, so $s(B)\ge w$. On the other hand,
by \cite[Prop.~11.14]{Ol93} we have $\nl(B)=k(p-1,w)$. By
Lemma~\ref{lem:Olsson} we find $\nl(B)=k(p-1,w)< p^w\le p^{s(B)}$ as claimed.
\par
Next let $B$ be a $p$-block of $\fA_n$, covered by a block $\hat B$ of
$\fS_n$ of weight $w$. First assume that $p>2$. Then the defect groups of $B$
and $\hat B$ agree. If $1\le w<p$, then $\hat B$ has elementary Abelian
defect groups, whence $s(B)=w$, while $\nl(B)< k(B)\le k(\hat B)$ by
\cite[Prop.~4.10]{Ol90}. So the claim follows from the already proven result
for $\fS_n$. So we have $p\le w$. Then clearly still
$\nl(B)\le 2\nl(\hat B)=2k(p-1,w)$. Now again $2k(p-1,w)\le 2(p-1)^w$
unless $p=3$ and $w\le6$, and an easy estimate shows that $2(p-1)^w< p^w$
whenever $3\le p\le w$. The finitely many cases when $p=3$ and $w\le6$ can be
checked easily. If $p=2$ then by \cite[Prop.~12.9]{Ol93} we have that
$\nl(B)=\nl(\hat B)$ if $w$ is odd and $\nl(B)=\nl(\hat B)+k(1,w/2)$ if $w$
is even. Moreover, $B$ satisfies $s(B)\ge s(\hat B)-1$. When $w=2$ the defect
groups are elementary Abelian of order~4, while $\nl(B)=3$. When $w=3$ then
$\nl(B)=w(1,3)=3$ and $|D|=2^3$, so $s(B)\ge2$. Finally, if $w\ge4$ then
$$\nl(B)\le \nl(\hat B)+k(1,w/2)=k(1,w)+k(1,w/2)\le 2^w\le 2^{s(B)}.$$
\par
By Lemma~\ref{lem:min} the claim then also holds for the 2-fold covering groups
of $\fA_n$ and $\fS_n$. Now assume that $p>2$ and let $B$ be a faithful
$p$-block of a 2-fold covering group of $\fS_n$, of weight $w$. Then by
\cite[Prop.~13.17]{Ol93} we have that
$$\nl(B)=\begin{cases} k(t,w)& \text{ if $w$ is even},\\
                     2k(t,w)& \text{ if $w$ is odd},\end{cases}$$
with $t=(p-1)/2$, while on the other hand $s(B)\ge w$. Thus, we may
conclude as before. Finally, any faithful block of the 2-fold cover of
an alternating group has exactly the same invariants as some faithful
block of a 2-fold covering of some symmetric group of the same weight (see
\cite[Rem.~13.18]{Ol93}), so the claim here follows from the previous
considerations.
\end{proof}

\begin{prop}   \label{prop:spor}
 Let $G$ be such that $[G,G]$ is quasi-simple and $S\le G/Z(G)\le\Aut(S)$ for
 some sporadic simple group $S$ or $S=\tw2F_4(2)'$. Then
 Conjecture~\ref{conj:main} holds for all $p$-blocks $B$ of $G$.
 It holds in strong form unless possibly when $p=3$ and defect groups of $B$
 are extraspecial of order~27.
\end{prop}

\begin{proof}
This statement can be essentially verified by computer on the known ordinary
character tables of the groups in question together with the known (lower
bounds on) the $p$-ranks as given in \cite[Tab.~5.6.1]{GLS}. The only blocks $B$
with defect at least~3 whose defect groups are not Sylow $p$-subgroups are
a 2-block of the third Conway group $Co_3$ of defect~3 with $\nl(B)=5$, and a
2-block of the Lyons group $Ly$ of defect~7 with $\nl(B)=8$. The former block
is known to have elementary Abelian defect groups. In the latter case, the
defect groups have index~2 in a Sylow $2$-subgroup, which is of rank~4, so the
inequality in Conjecture~\ref{conj:main} holds as well. We claim that in fact
a defect group $D$ of $B$ has rank~4 in this case as well. Indeed, $G=Ly$ has a
unique class of involutions with representative $t$, say, with centraliser
$C=C_G(t)$ a double cover of $\fA_{11}$. There is a Brauer correspondent of $B$
for $C$ which has the same defect group $D$ of order~$2^7$.
Let $D_1= D/\langle t\rangle$. Then $D_1$ is a defect group of a 2-block $b$ of
$\fA_{11}$ which is covered by a block of weight 4 of $\fS_{11}$ whose defect
group intersected with $\fA_{11}$ has rank~4.
\end{proof}

\begin{rem}   \label{rem:27}
The principal 3-blocks of the Tits group $\tw2F_4(2)'$, of the Rudvalis
group $Ru$ and of the fourth Janko group $J_4$, as well as a further
non-principal 3-block of $J_4$, have extra-special defect group
$D\cong 3^{1+2}_+$ of order~27 and each possesses 9 irreducible Brauer
characters, so they provide examples of blocks $B$ with $\nl(B)=p^{s(B)}$.
These (and related examples in automorphism groups as well as examples in
groups $\tw2F_4(q^2)$ with $q^2\ge8$ also for $p=3$ and with sectional rank
equal to~2, see Proposition~\ref{prop:smallexc}) are the only cases of positive
defect known to us where equality in Conjecture~\ref{conj:main} occurs.
\end{rem}

\section{Groups of Lie type} \label{sec:Lie}

\subsection{Defining characteristic}
We first deal with the defining characteristic case.

\begin{prop}   \label{prop:Liedef}
 Let $G$ be a finite quasi-simple group of Lie type in characteristic~$p$.
 Then Conjecture~\ref{conj:main} holds in strong form for all $p$-blocks of $G$.
\end{prop}

\begin{proof}
By Lemma~\ref{lem:min} we need not consider covering groups with $p$ dividing
the order of the centre. But then we may assume that $G$ is the universal
$p'$-covering group of its simple factor. In that case,
$G$ is obtained as $\bG^F$ where $\bG$ is a simple, simply connected linear
algebraic group over an algebraic closure of $\FF_p$ and $F:\bG\rightarrow\bG$
is a Steinberg endomorphism. By a result of Humphreys $G$ has exactly one
$p$-block of defect zero, containing the Steinberg character, and all other
$p$-blocks of $G$ have full defect. First assume that $G$ is not of twisted
type. Then Sylow $p$-subgroups of $G$ have $p$-rank at least $rf$, where
$q=p^f$ is the size of the underlying field of $G$ and $r$ denotes the rank
of $\bG$, see \cite[Tab.~3.3.1]{GLS}.
On the other hand, the simple modules of $G$ in characteristic~$p$ are
parametrised by $q$-restricted weights, of which there exist exactly $q^r$.
One of them belongs to the Steinberg character, so $\nl(B)\le q^r-1$ for all
blocks $B$ of positive defect. Thus
$$\nl(B)\le q^r-1 <q^r\le p^{s(B)}.$$
Next assume that $G$ is twisted, but not of Ree or Suzuki type. Then again by
loc.~cit., the $p$-rank of $G$ is at least $rf$, where $q=p^f$ is the absolute
value of the eigenvalues of $F$ on the character group of a
maximal torus of $\bG$, unless $p=2$ and $G=\SU_3(q)$. But note that the
quotient of a Sylow $p$-subgroup of $\SU_3(q)$ by the highest root subgroup
is elementary Abelian, so has $p$-rank $2f$. The number of simple modules
of $G$ again equals $q^r$. So we may conclude as before.
Finally, the groups $\tw2B_2(q^2)$, $^2G_2(q^2)$ and $\tw2F_4(q^2)$ ($q^2\ge8$)
have $p$-ranks $f$, $2f$, $5f$ respectively, where $q^2=p^f$, and the same
argument applies.
\end{proof}

\begin{rem}
The proof shows that for all primes $p\ge5$ there exist $p$-blocks $B$ of
simple groups with extra-special defect group $p^{1+2}$ in which
$\nl(B)=p^2-1=p^{s(B)}-1$, so the bound in the strong form of
Conjecture~\ref{conj:main} cannot be improved to $p^{s(B)}-2$ for any prime.
Indeed, for $p\equiv2\pmod3$ the principal $p$-block of $\SL_3(p)$ is as
claimed, while for $p\equiv1\pmod3$ we may take the principal $p$-block of
$\SU_3(p)$.
\end{rem}

Concerning nearly simple groups, we make the following observations:

\begin{cor}   \label{cor:graph-field}
 Let $G$ be a finite quasi-simple group of Lie type in characteristic~$p$
 and $\si$ a field, graph or graph-field automorphism of $G$. Then
 Conjecture~\ref{conj:main} holds in strong form for all $p$-blocks of the
 extension $\tilde G:=G\langle\si\rangle$ of $G$.
\end{cor}

\begin{proof}
Let $B$ be a $p$-block of $G$ of positive defect and $\tilde B$ a $p$-block
of $\tilde G$ covering $B$. As explained in the proof of
Proposition~\ref{prop:Liedef} the simple modules in $B$ are labelled by
(a subset of the) $q$-restricted weights for $G$, which naturally can be given
the structure of an elementary Abelian $p$-group $P$. Now the action of $\si$
on the simple modules is induced by an action of $\si$ on the associated
weights, which may be considered as a faithful action of $\si$ on $P$, where
$r$ is the rank of $G$. Note that we may assume that $\si$ has order prime to
$p$. Then by Lemma~\ref{lem:k(GV)}, the number of characters of
$\tilde G$ above those in $B$ is at most $q^r$ minus the number of
characters lying above the Steinberg weight (which belongs to a different
block). Since $s(\tilde B)\ge s(B)$ the claim now follows from the arguments
in the proof of Proposition~\ref{prop:Liedef}.
\end{proof}

\subsection{Exceptional coverings and small rank}
We first consider the finitely many exceptional covering groups.

\begin{prop}   \label{prop:excover}
 Conjecture~\ref{conj:main} in strong form holds for all blocks of exceptional
 covering groups of finite simple groups of Lie type and all primes.
\end{prop}

\begin{proof}
Again, the statement can be essentially verified by inspection on the known
ordinary character tables of the groups in question together with obvious
lower bounds on the $p$-ranks. All blocks are either of maximal
defect, or of defect at most~2.
\end{proof}

We next consider exceptional groups of small rank.

\begin{prop}   \label{prop:smallexc}
 Let $B$ be a $p$-block of a finite quasi-simple group $G$ of Lie type
 $\tw2B_2,\tw2G_2$, $G_2,\tw3D_4$ or $\tw2F_4$. Then Conjecture~\ref{conj:main}
 holds for $B$. It holds in strong form unless $B$ is the principal 3-block
 of $\tw2F_4(q^2)^({}'{}^)$. 
\end{prop}

\begin{proof}
The exceptional covering groups were dealt with in
Proposition~\ref{prop:excover}.
Next, if $p$ is the defining prime for $G$ the result was already proved in
Proposition~\ref{prop:Liedef}. Furthermore, we may assume that the defect
groups of $B$ are non-cyclic. Thus $G$ is of type $G_2, \tw3D_4$ or $\tw2F_4$,
or $p=2$ and $G={}^2G_2(q^2)$. In the latter case the Sylow 2-subgroups are
elementary Abelian of order~8, while $\nl(B)=3$ for all 2-blocks of defect at
least~2 by \cite[p.~74]{Wa66}.
\par
For $G_2(q)$ the numbers $\nl(B)$ in the non-cyclic defect case were determined
by Hiss and Shamash \cite{Hi89,HS90,HS92}: The principal block $B$ has
$\nl(B)\le7$ in all cases, while the $p$-rank is equal to~2 when $p\ge3$,
and~3 for $p=2$. All other blocks $B$ have $\nl(B)\le3$. For $\tw3D_4(q)$,
then \cite[Thm.~A]{Ge93} shows that $\nl(B)\le7$ for the principal $p$-block
when $p$ is good (so when $p\ne2$), while  non-cyclic Sylow $p$-subgroups
obviously have
$p$-rank at least~2. When $p=2$ then again $\nl(B)=7$ by \cite[Thm.~3.1]{Hi07}
and Sylow 2-subgroups have rank~3. For all other blocks the validity of
Conjecture~\ref{conj:main} is immediate. Finally, for $\tw2F_4(q^2)$,
$q^2\ge8$, the numbers $\nl(B)$ for primes $p\ge3$ have been
determined in \cite{Hi11}. From this the inequality can be checked, and as
for the case of $\tw2F_4(2)'$ in Proposition~\ref{prop:spor} we obtain
equality (only) when $B$ is the principal 3-block. 
\end{proof}

\begin{thm}   \label{thm:SLSU}
 Let $B$ be a $p$-block of a quasi-simple group $\SL_n(q)$ ($n\ge2$)
 or $\SU_n(q)$ ($n\ge3)$ in characteristic~$r\ne p$. Then
 Conjecture~\ref{conj:main} holds for $B$.
\end{thm}

\begin{proof}
Our argument crucially relies on the description of $p$-blocks of $\GL_n(q)$
and $\GU_n(q)$ by Fong and Srinivasan \cite{FS82}, as well as on the result of
Geck \cite[Thm.~A]{Ge93} that for any finite group of Lie type $G$ and any
semisimple $p'$-element $s\in G^*$ the set $\cE(G,s)$ forms a basic set for
$\cE_p(G,s)$ whenever $p$ is a good prime not dividing the order of the group
of components of the center of the underlying algebraic group.
\par
We first show our claim for unipotent $p$-blocks of $G=\GL_n(q)$. Note that the
centre of $\GL_n$ is connected. Thus the unipotent characters form a basic set
in any unipotent block $B$, and the number $\nl(B)$ is just the number
of unipotent characters in $B$. Let $d$ denote the multiplicative order of $q$
modulo~$p$; in particular we then have $d\le p-1$. Let $B$ be a unipotent
$p$-block of $G$. Then the unipotent characters in $B$ are those labelled by
partitions $\mu$ of $n$ with a fixed $d$-core $\lambda\vdash n-wd$
for a suitable weight~$w$, see \cite{FS82}, and \cite[Thm.~21.14]{CE} for
$p=2$. The number of such partitions equals $k(d,w)$, so $\nl(B)=k(d,w)$.
On the other hand, by \cite{FS82} a defect group $D$ of $B$ contains an
elementary Abelian subgroup of order~$p^w$, so $s(B)\ge w$ (respectively for
$p=2$ we have $d=1$ and $B$ is the principal block). Now we have
$k(d,w)< (d+1)^w$ by Lemma~\ref{lem:Olsson} and as $d\le p-1$, our claim
follows for unipotent blocks.
\par
Next let $B$ be an arbitrary $p$-block of $\GL_n(q)$. By the theorem of
Bonnaf\'e--Rouquier \cite[Thm.~10.1]{CE}, $B$ is Morita equivalent to a
unipotent block of
$C_{G^*}(s)$, with isomorphic defect groups. (Note that here
$G^*\cong G=\GL_n(q)$.) But
$$C_{G^*}(s)=G_1\times\cdots\times G_r\quad
    \text{ with }\quad G_i\cong\GL_{n_i}(q^{f_i})
    \text{ for certain } n_1f_1+\ldots+n_rf_r=n,$$
so $B$ is the product of unipotent blocks $B_i$ of $G_i$, and similarly for
the defect groups. Thus our previous considerations show that $B$ satisfies
the assertion of Conjecture~\ref{conj:main}.
\par
Now consider the case when $G=\SL_n(q)$. Embed $G$ as a subgroup of
$\tilde G=\GL_n(q)$ in the natural way. Let $B$ be a $p$-block of $G$ in
series $\cE_p(G,s)$. Let $\tilde s\in\tilde G^*$ be a preimage of $s$
under the induced epimorphism $\tilde G^*\rightarrow G^*$ and $\tilde B$ be
a $p$-block of $\tilde G$ in $\cE_p(\tilde G,\tilde s)$ covering $B$. By the
previous case, the claim holds for $\tilde B$. But then it also holds for
$B$ whenever $p{\not|}(q-1)$, since on the one hand side $|Z(G)|$ divides
$q-1$, hence is prime to $p$ and so $\Irr(B)\cap\cE(G,s)$ is a basic
set for $B$, and on the other hand $|\tilde G:G|=q-1$ is prime to $p$ in this
case, so the defect groups of $B$ and $\tilde B$ agree.
\par
So now assume that $p|(q-1)$, whence $d=1$. As $\tilde G/G$ is cyclic, $B$
satisfies $s(B)\ge s(\tilde B)-1$, and furthermore $\nl(B)\le n \nl(\tilde B)$.
As above write
$$C_{\tilde G^*}(\tilde s) =G_1\times\cdots\times G_r$$
and let $B_i$ denote the unipotent $p$-block of $G_i$ such that $\tilde B$ is
Morita equivalent to $B_1\otimes\cdots\otimes B_r$. Since $p|(q-1)$, each
$B_i$ is the principal block of $G_i$, and thus contains $k(1,n_i)$ unipotent
characters, while $s(B_i)\ge n_i$. It is straightforward to check that
the required inequality holds unless possibly when $n_i=2$, $p\le3$. For
$n_i=2$ and $p=3$ the Sylow $p$-subgroups of $G_i$ are cyclic, while for
$n_i=p=2$ we have $\nl(B_i)=3$ and $|D_i|=8$.
\par
The preceding arguments carry over almost word by word to $\SU_n(q)$. Again,
we first deal with $G=\GU_n(q)$. Note that centralisers of semisimple elements
$s\in G^*\cong\GU_n(q)$ have the form
$$C_{G^*}(s)= G_1\times\cdots\times G_r\quad\text{ with} \quad
  G_i\cong\GL_{n_i}((-q)^{f_i})$$
for suitable $n_1f_1+\ldots+n_rf_r=n$, where $\GL_m(-u)$ with $-u<0$ is to be
interpreted as $\GU_m(u)$. Then, setting $d$ now to be the order of $-q$
modulo~$p$ we can argue as before. For $p=2$ we again use the fact that all
unipotent characters lie in the principal 2-block, see \cite[Thm.~21.14]{CE}.
We then deal with $\SU_n(q)$ via the regular embedding
$\SU_n(q)\hookrightarrow\GU_n(q)$. Here, the critical case is when $p|(q+1)$,
but then again all unipotent characters of $G_i\cong\GL_{n_i}((-q)^{f_i})$ lie
in the principal block and the same estimates as above yield the claim.
\end{proof}

\subsection{Unipotent blocks}
While our results for the classes of quasi-simple groups considered so far are
complete, it seems that at present the knowledge on blocks of other groups of
Lie type in non-defining characteristic is not yet strong enough to check
Conjecture~\ref{conj:main} in all cases; we only obtain partial results:

\begin{thm}   \label{thm:Lieunip}
 Let $B$ be a unipotent $p$-block of a finite quasi-simple group $G$ of
 classical Lie type in characteristic~$r\ne p$. Then Conjecture~\ref{conj:main}
 holds for $B$.
\end{thm}

\begin{proof}
We will use throughout the result of Geck \cite[Thm.~A]{Ge93} that the
unipotent characters form a basic set for the unipotent $p$-blocks of any
finite group of Lie type for which $p$ is a good prime and such that $p$ does
not divide the order of the group of components of the center of the underlying
algebraic group. Then in particular in any unipotent block $B$, the number
$\nl(B)$ is just the number of unipotent characters in $B$.
\par
First consider $G=\Sp_{2n}(q)$ and $G=\SO_{2n+1}(q)$. If $p>2$ the
unipotent characters form a basic set for the
unipotent $p$-blocks of $G$. Let $d$ denote the order of $q$ modulo~$p$.
We need to distinguish two cases. First assume that $d$ is odd. Then the
unipotent blocks of $G$ are indexed by Lusztig symbols $\cS$ without $d$-hooks,
called $d$-cores, and a unipotent character $\chi$ lies in the block $B$
corresponding to $\cS$ if and only if it has $d$-core $\cS$. The number of such
characters equals the number of irreducible characters of the corresponding
relative Weyl group
(see \cite[Thm.~3.2]{BMM}), of type $C_{2d}\wr \fS_w$, where $w$ is the weight
of $B$. Thus $\nl(B)=k(2d,w)$. Defect groups of $B$ then contain an elementary
Abelian subgroup of order $p^w$, so we have $s(B)\ge w$. Now note that $d$
divides $p-1$, and as $d$ is odd this forces $d\le(p-1)/2$. By
Lemma~\ref{lem:Olsson} we have $\nl(B)=k(2d,w)\le k(p-1,w)<p^w\le p^{s(B)}$.
On the other hand, if $d$ is even, then again using \cite[Thm.~3.2]{BMM} we
have that $\nl(B)=k(C_d\wr \fS_w)=k(d,w)$ with $d\le p-1$, while still
$s(B)\ge w$, so we can conclude as before. \par
All of the above assertions remain valid for unipotent blocks of groups of
type $D_n$ or $\tw2D_n$, unless $B$ contains the unipotent characters labelled
by symbols whose $d$-core (when $d$ is odd) or $e$-cocore (when $d=2e$ is even)
is a so-called degenerate symbol. First assume that $d$ is odd. Then the
relative Weyl group is the reflection subgroup $G(2d,2,w)$ of $C_{2d}\wr \fS_w$
of index~2, defect groups of $B$ still have rank at least~$w$, and
$\nl(B)< (2d+1)^w\le p^w$ by Lemma~\ref{lem:degenerate}. We may argue entirely
similar when $d$ is even.
\par
If $G$ is of type $B_n$, $C_n$ or $D_n$ and $p=2$, then let
$G\hookrightarrow\tilde G$ denote a regular embedding. By \cite[Thm.~13]{CE93}
all unipotent characters of $\tilde G$ lie in the principal 2-block $\tilde B$
of $\tilde G$. Furthermore, by \cite[Prop.~2.4]{Ge94} the number
$\nl(\tilde B)$ equals the number of unipotent classes of $\tilde G$. Upper
bounds for these are given in Lemma~\ref{lem:uclass}. Now first assume that we
are in type $B_n$ or $C_n$. Then $\tilde G$ induces on $G$ an outer
automorphism of order~2. Thus if $B$ is the principal 2-block of $G$, then
$\nl(B)\le 2 \nl(\tilde B)$. Lower bounds for the sectional 2-rank of $G$
are given in Table~\ref{tab:2rank}.

\begin{table}[htbp]
\caption{Lower bounds for sectional 2-ranks, $q$ odd}   \label{tab:2rank}
\[\begin{array}{|r|cccc|}
\hline
 G& B_n(q)& C_n(q)& D_n(q)& \tw2D_n(q)\\
\hline
  & 2n& 2n& 2n-1& 2n-1\\
\hline
\end{array}\]
\end{table}
Indeed, according to \cite[Table~2.5]{BHR}, the simple orthogonal group
$\OO_n^{(\pm)}(q)$, with $q$ odd, contains a subgroup $2^{n-1}.\fA_n$ (inside
the natural subgroup $\OO_1(q)\wr\fA_n$), so has 2-rank at least $n-1$.
Similarly, $\Sp_{2n}(q)$ contains a direct product $\Sp_2(q)^n$, each
factor of which has sectional 2-rank~2.
Then for $\bG$ of type $C_n$, as $2^{2n}>2\cdot 2^{n+\flsq{n}}$
for $n\ge3$, the desired result follows for $B$. When $n=2$ then $\tilde G$
has 5 unipotent classes and again our inequality is satisfied. For $\bG$ of
type $B_n$ ($n\ge3$) the sectional 2-rank of $G$ is at least~$2n$ by
Table~\ref{tab:2rank}, and again we are done unless $n\in\{3,4\}$. In the
latter two cases, the actual number of unipotent classes of $\tilde G$ is
$10, 21$ respectively, smaller than $2^{2n}/2$.
\par
For $\bG$ of type $D_n$ ($n\ge4$), $\tilde G$ induces on $G$ a group of
automorphisms of order dividing~4. The principal 2-block $\tilde B$ of
$\tilde G$ has $\nl(\tilde B)\le 2^{n+\flsq{2n}}$ by Lemma~\ref{lem:uclass}.
Thus the claim follows by Table~\ref{tab:2rank} unless $4\le n\le6$. In the
latter cases, $\tilde G$ has $13, 18, 37$ unipotent classes, respectively (and
even fewer in the twisted case), which is smaller than $2^{2n-1}/4$.
\end{proof}

The following estimates were used in the previous proof:

\begin{lem}   \label{lem:degenerate}
 Let $w,d\ge1$. Then $|\Irr(G(2d,2,w))|< (2d+1)^w$.
\end{lem}

\begin{proof}
The reflection group $G(2d,1,w)$ is the wreath product $C_{2d}\wr\fS_w$,
hence its irreducible characters are parametrised by $2d$-tuples of partitions
of $w$, of which there are $k(2d,w)$. Such a character splits upon restriction
to the normal reflection subgroup $G(2d,2,w)$ if the parametrising $2d$-tuple
$\la$ has a symmetry of order~2, that is, if $w$ is even and $\la$ is the
concatenation of twice a $d$-tuple $\mu$ of partitions of $w/2$, and it
restricts irreducibly else. So by Lemma~\ref{lem:Olsson}
$|\Irr(G(2d,2,w))|=\hlf k(2d,1,w)<(2d+1)^w$ when $w$ is odd, and
$$\begin{aligned}
  |\Irr(G(2d,2,w))|&=\hlf(k(2d,1,w)-k(d,1,v))+2k(d,1,v)\\
                  &= \hlf k(2d,1,w) + \frac{3}{2}k(d,1,v)
                  \le \hlf(2d+1)^w+\frac{3}{2}(d+1)^{v}<(2d+1)^w
\end{aligned}$$
if $w=2v$ is even.
\end{proof}

\begin{lem}   \label{lem:uclass}
 Let $\bG$ be a simple algebraic group in odd characteristic of adjoint type
 $B_n$, $C_n$ or $D_n$ respectively, and $F:\bG\rightarrow\bG$ a Frobenius
 endomorphism. Then $G=\bG^F$ has at most $2^{n+\flsq{2n+1}}$, respectively
 $2^{n+\flsq{n}}$, $2^{n+\flsq{2n}}$ unipotent
 conjugacy classes.
\end{lem}

\begin{proof}
The unipotent classes of $\bG$ and $G$ are described for example in
\cite[\S13.1]{Ca85}.
First assume that $\bG$ is of type $C_n$. Its unipotent classes are parametrised
by pairs of partitions $(\al,\beta)$ of $n$, where $\beta$ has distinct parts.
Thus, using Lemma~\ref{lem:Olsson} and direct computation for $n\le 6$ there
are at most $2^n$ unipotent classes in $\bG$. Each class $C$ of $\bG$ splits
into $a(C)$ classes in $G$, where $a(C)$ is the order of the component group of
the centraliser of any element $u\in C$. This component group has order
$2^{n(C)}$, where $n(C)$ is at most the number of even $i$ such that
$(\al,\beta)$ has a part of length~$i$. This becomes maximal if $(\al,\beta)$
has parts of lengths $2,4,6,\ldots$. Since $\sum_{i=1}^k 2i=k(k+1)$ we see that
$n(C)\le \flsq{n}$, whence the result.
\par
For $\bG$ of type $B_n$, the unipotent classes are parametrised by pairs of
partitions $(\al,\beta)$ such that $2|\al|+|\beta|=2n+1$, where $\beta$ has
distinct odd parts. To any such pair we associate a pair of partitions
$(\al,\beta')$ of $n$ as follows: if $\beta=(\beta_1\le\beta_2\le\ldots)$ then
$\beta'=((\beta_1-1)/2\le(\beta_2+1)/2\le(\beta_3-1)/2\le\ldots)$ (note that
$\beta$ necessarily has an odd number of parts, so $|\beta'|=(|\beta|-1)/2$).
Clearly this map is injective,
so again the number of unipotent classes of $\bG$ is at most $2^n$. The
component group of the centraliser of a unipotent element in the class $C$
parametrised by $(\al,\beta)$ has order $2^{n(C)}$, where $n(C)$ is at most
the number of distinct odd parts of $(\al,\beta)$. As $\sum_{i=1}^k(2i-1)=k^2$
we find $n(C)\le\flsq{2n+1}$ and hence the claim.
\par
Finally, for $\bG$ of type $D_n$ the unipotent classes are parametrised by
pairs of partitions $(\al,\beta)$ with $2|\al|+|\beta|=2n$, where $\beta$ has
distinct odd parts. If $\beta$ is empty and all parts of $\al$ are even, then
$(\al,\beta)$ parametrises two unipotent classes. Arguing as in the previous
case we see that there are at most
$2^n-2^{n/2}$ non-degenerate cases, and $2^{n/2+1}$ degenerate ones. The
component group of the corresponding centraliser has order $2^{n(C)}$, where
$n(C)$ is at most the number of distinct odd parts. Hence we get a factor of
at most $2^{\flsq{2n}}$ for non-degenerate classes, while degenerate classes
have connected centralisers. This yields at most
$$(2^n-2^{n/2})2^{\flsq{2n}}+2^{n/2+1}
  =2^{n+\flsq{2n}}-2^{n/2+\flsq{2n}}+2^{n/2+1}\le 2^{n+\flsq{2n}}$$
unipotent classes for $G$.
\end{proof}

\par
We now turn to the groups of exceptional type. The small rank groups have
already been considered in Proposition~\ref{prop:smallexc}.

\begin{prop}   \label{prop:unipexc}
 Let $B$ be a unipotent $p$-block of a finite quasi-simple group $G$ of
 exceptional Lie type in characteristic~$r\ne p$ of rank at least~4.
 If $p$ is bad for $G$, then assume that $B$ is the principal block.
 Then Conjecture~\ref{conj:main} holds for $B$.
\end{prop}

\begin{proof}
For the remaining types $F_4,E_6,\tw2E_6,E_7$ and $E_8$ first assume that $p$
is a good prime (so $p\ge5$, and $p\ge7$ for type $E_8$).
Let $d$ denote the order of $q$ modulo~$p$. The numbers $\nl(B)$ for the
principal $p$-block and the $p$-ranks of Sylow $p$-subgroups are then as given
in Table~\ref{tab:l(b)exc} by \cite[Thm.~3.2 and Table~3]{BMM}. (The entries
``$-$'' signify that the Sylow $p$-subgroups are cyclic.) It ensues that the
required inequality
is satisfied in all cases, even in the strict form. More generally, the
non-principal unipotent blocks are described in \cite[Tab.~1]{BMM}. An easy
check, similar to that for the principal blocks, shows that the inequality
in its strict form also holds for those blocks.
\par
If $p$ is bad for $G$, then the total number of simple modules in unipotent
$p$-blocks of $G$ is as given in Table~\ref{tab:l(b)bad}, see
\cite[\S4.1]{DGHM}. Note that the assumption in loc.~cit.~on the underlying
field size of $G$ is now unnecessary since by results of Lusztig cuspidal
character sheaves are always 'clean', see \cite{Lu12}. Table~\ref{tab:l(b)bad}
also gives obvious lower
bounds for the $p$-ranks in the respective cases (obtained by looking at
suitable maximal tori of the groups in question). Furthermore, for $p=2$ in
$F_4(q)$ note that the sectional 2-rank is at least~8, since $F_4(q)$ contains
a central product $A$ of $4$ commuting $A_1$-type subgroups, see
\cite[Tab.~4.10.6]{GLS}. Then $A/Z(A)$ contains a direct product
$\prod_{i=1}^4\PSL_2(q)$; since $\PSL_2(q)$ has 2-rank~2, the claim follows.
\end{proof}

\begin{table}[htbp]
\caption{$\nl(B)$ for principal blocks and $p$-ranks}   \label{tab:l(b)exc}
\[\begin{array}{|r|rrrrr|}
\hline
 (\nl(B):s)& F_4& E_6& \tw2E_6& E_7\ & E_8\ \\
\hline
   1& 25:4& 25:6& 25:4& 60:7& 112:8\\
   2& 25:4& 25:4& 25:6& 60:7& 112:8\\
 d=\quad\ 3&  21:2& 24:3& 21:2& 48:3& 102:4\\
   4&   -&  16:2& 16:2& 16:2&  59:4\\
   6& 21:2&  21:2& 24:3& 48:3& 102:4\\
\hline
\end{array}\]
\end{table}

\begin{table}[htbp]
\caption{$\nl(B)$ and lower bounds on $p$-ranks for bad primes}   \label{tab:l(b)bad}
\[\begin{array}{|r|rrrr|}
\hline
 (\nl(B):s)& F_4& ^{(2)}\!E_6& E_7\ & E_8\ \\
\hline
   2& 28:4& 27:6& 64:7& 131:8\\
 p=\quad\ 3& 35:4& 28:4& 72:7& 150:8\\
   5&    -&    -&    -& 162:4\\
\hline
\end{array}\]
\end{table}

\subsection{Towards general blocks}
We give some further partial results for general blocks of finite
quasi-simple groups of Lie type.

\begin{thm}
 Let $\bG$ be simple with Frobenius endomorphism $F$, and let $p\ge3$ be a
 good prime for $\bG$. Then Conjecture~\ref{conj:main} holds for all
 $p$-blocks of $\bG^F$ if it holds for all non-unipotent quasi-isolated
 $p$-blocks of all $F$-stable Levi subgroups $\bL\le \bG$.
\end{thm}

\begin{proof}
Let $B$ be a $p$-block of $\bG^F$, in the Lusztig series of the semisimple
$p'$-element $s\in\bG^{*F}$. Let $\bL^*\le\bG^*$ be the minimal $F$-stable
Levi subgroup containing $C_{\bG^*}(s)$ (this is uniquely determined since the
intersection of two Levi subgroups containing a common maximal torus is again
a Levi subgroup). Then $s$ is quasi-isolated in $\bL^*$. Let $\bL\le\bG$
be dual to $\bL^*$. According to the theorem of Bonnaf\'e, Dat and Rouquier
\cite[Thm.~7.7]{BDR} Lusztig induction $R_{\bL}^\bG$ induces a Morita
equivalence between the $p$-blocks in $\cE_p(\bL^F,s)$ and those in
$\cE_p(\bG^F,s)$, which sends $\Irr(B_1)$ bijectively to $\Irr(B)$ for some
$p$-block $B_1$ contained in $\cE_p(\bL^F,s)$ and that preserves defect groups.
In particular $\nl(B)=\nl(B_1)$ and $s(B)=s(B_1)$, and so
Conjecture~\ref{conj:main} holds for $B$ if it holds for the quasi-isolated
block $B_1$ of the Levi subgroup $\bL^F$. 
\par
If $B_1$ is a unipotent block, then since unipotent characters are insensitive
to isogeny types, our conjecture follows from our proof for unipotent blocks
given in  Theorems~\ref{thm:SLSU}, \ref{thm:Lieunip} and
Proposition~\ref{prop:unipexc}.
\end{proof}


\end{document}